\documentclass[reqno]{amsart}
\usepackage{amsmath, amssymb, amsthm, epsfig}
\usepackage{hyperref, latexsym}
\usepackage{url}
\usepackage[mathscr]{euscript}
\usepackage{mathabx}

\usepackage{pgfplots}
\pgfplotsset{width=8cm,compat=1.9}

\usepackage{color}
\usepackage{setspace}

\usepackage{mathtools}
\mathtoolsset{showonlyrefs}

\def\today{\ifcase\month\or
  January\or February\or March\or April\or May\or June\or
  July\or August\or September\or October\or November\or December\fi
  \space\number\day, \number\year}

 \newtheorem{theorem}{Theorem}
  
 \newtheorem{lemma}[theorem]{Lemma}
 \newtheorem{proposition}[theorem]{Proposition}
 
 \theoremstyle{definition}

 \theoremstyle{remark}

 \newcommand{\mc}{\mathcal}
 
 \newcommand{\B}{\mc{B}}
 
 \newcommand{\D}{\mc{D}}
 \newcommand{\E}{\mc{E}}

 \newcommand{\C}{\mathbb{C}}
 \newcommand{\R}{\mathbb{R}}

 \newcommand{\Z}{\mathbb{Z}}

 \newcommand{\p}{\varphi}
 
 \newcommand{\hh}{\tfrac12}

  \renewcommand{\d}{\text{\rm d}}
 \newcommand{\du}{\text{\rm d}u}

 \newcommand{\dx}{\text{\rm d}x}

\newcommand{\ov}{\overline}
\renewcommand{\H}{\mc{H}}
\newcommand{\im}{{\rm Im}\,}
\newcommand{\re}{{\rm Re}\,}

\renewcommand{\D}{\Delta}
\newcommand{\ga}{\gamma}
\newcommand{\la}{\lambda}
\newcommand{\ft}{\widehat}

\newcommand{\al}{\alpha}
\newcommand{\be}{\beta}

\newcommand{\si}{\sigma}

\begin{document}
\title[Estimates for the Riemann zeta-function]{Bounding the log-derivative of the zeta-function}
\author[Chirre and Gon\c{c}alves]{Andr\'es Chirre and Felipe Gon\c{c}alves}
\address{Hausdorff Center for Mathematics, 53115 Bonn, Germany.}
\email{goncalve@math.uni-bonn.de}
\address{Department of Mathematical Sciences, Norwegian University of Science and Tech-$\text{ }$ $\text{ }$ $\text{ }$ nology, NO-7491 Trondheim, Norway.}
\email{carlos.a.c.chavez@ntnu.no}
\allowdisplaybreaks
\numberwithin{equation}{section}

\begin{abstract}
Assuming the Riemann hypothesis we establish explicit bounds for the modulus of the log-derivative of Riemann's zeta-function in the critical strip.
\end{abstract}

\maketitle


\section{Introduction}

Let $\zeta(s)$ be the Riemann zeta-function. In this paper we are interested in its log-derivative
$$
\dfrac{\zeta'}{\zeta}(s) = \sum_{n\geq 1} \dfrac{\Lambda(n)}{n^{s}} \quad (\re s>1)
$$
and its growth behaviour in the strip $1/2<\re s  < 1$ (above $\Lambda(n)$ is the von Mangoldt function). Let $\rho$ denote the zeros of $\zeta(s)$ in the critical strip. The Riemann hypothesis (RH) states that the zeros are aligned: $\rho=\hh+i\gamma$ with $\gamma\in\R$. Assuming RH, a classical estimate for the log-derivative of $\zeta(s)$ (see \cite[Theorem 14.5]{Tit}) establishes that
\begin{equation}  \label{classicalresult}
\dfrac{\zeta'}{\zeta}(\si+it)=O\big((\log t)^{2-2\sigma}\big),
\end{equation}
uniformly in $\hh+\delta\leq \sigma\leq 1-\delta$, for any fixed $\delta>0$. The purpose of this paper is to establish this bound in explicit form.


\begin{theorem}\label{thm:0}
Assume RH. Then
\begin{align}  \label{main}
\bigg|\dfrac{\zeta'}{\zeta}(\si+it)\bigg|\leq \frac{B_\sigma}{\si(1-\si)} (\log t)^{2-2\sigma} + O\bigg(\dfrac{(\log t)^{2-2\sigma}}{(\sigma-\hh)(1-\si)^2\log\log t}\bigg),\end{align}
uniformly in the range
	\begin{align} \label{range}
	\dfrac{1}{2}+\dfrac{\la_0+c}{\log\log t}\leq \sigma\leq   1 - \frac{c}{\sqrt{\log \log t}}  \quad \text{and} \quad t\geq 3,
	\end{align}
for any fixed small $c>0$, where $\la_0=0.771\ldots$ is such that $2\lambda_0 \tanh(\lambda_0)=1$ and
$$
B_\si =\sqrt{\frac{(3\si^4-17\si^3+19\si^2+4\si-4)(-\si^2+3\si-1)}{\si(2-\si)}}.
$$
In particular
$$
\bigg|\dfrac{\zeta'}{\zeta}(\si+it)\bigg|\leq \left(\frac{B_\sigma}{\si(1-\si)} +o(1)\right)(\log t)^{2-2\sigma},  \quad \text{for} \quad \tfrac12+\delta \leq \si \leq 1-\delta.
$$
\end{theorem}
We believe that $\la_0$ is simply a by-product of our proof, although it is curious that such a number appears. It turns out that when $(\si-1/2)\log\log t$ is too small, our main technique delivers a bound of the form ${A_\si}{ (\log t )}/{\log \log t} $, however the calculations are lengthy and convoluted, and this it not the purpose of this note. Moreover, a conjecture of Ki \cite{Ki}, related to the distribution of the zeros of $\zeta'(s)$, states that the bound $O((\log t)^{2-2\si} )$ still holds in the range $\si \geq 1/2+c/\log t$, but this lies outside of what this technique can accomplish. Theorem \ref{thm:0} is derived by combining Theorem \ref{theorem1} and estimates for the real part of the log-derivative of $\zeta(s)$ obtained in \cite[Theorem 2]{CChiM}:
$$
\left|\re \dfrac{\zeta'}{\zeta}(\si+it)\right| \leq \bigg(\dfrac{-\si^2 + 3\si - 1}{\sigma(1-\sigma)}\bigg) (\log t)^{2-2\si} + O\left(\frac{(\log t)^{2-2\si}}{(\sigma-\hh)(1-\si)^2\log \log t}\right),
$$
uniformly in the range \eqref{range} (in fact $\la_0+c$ can be replaced by just $c$).


\begin{theorem} \label{theorem1} 
	Assume RH. Then
	$$
\left| \im\dfrac{\zeta'}{\zeta}(\sigma+it) \right| \leq \frac{C_\si}{\si(1-\si)} (\log t)^{2-2\sigma} + O\bigg(\dfrac{(\log t)^{2-2\sigma}}{(\sigma-\hh)(1-\si)^2\log\log t}\bigg),
	$$
	uniformly in the range \eqref{range}, where
	$$
	C_\si=\sqrt{\frac{2(-\si^2 + 5\si - 2)(-\si^2 + 3\si - 1)(-\si^2 + \si + 1)}{\si(2-\si)}}.
	$$
\end{theorem}

Theorem \ref{theorem1} is obtained using a known interpolation technique \cite[Section 6]{CChiM}. Essentially, to bound the asymptotic growth of $\im \tfrac{\zeta'}{\zeta}(s)$ one can bound instead its primitive $\log |\zeta(s)|$ (see \cite[Theorems ~1 and ~2]{CC}) and its derivative (Theorem \ref{maintheorem}). 
\smallskip

\begin{theorem} \label{maintheorem}
	Assume RH . Then
	\begin{equation} \label{second_result} 
	\re\bigg(\dfrac{\zeta'}{\zeta}\bigg)'(\sigma+it)\leq \bigg(\dfrac{-2\sigma^2+2\si+2}{\sigma(1-\sigma)}\bigg)\log\log t\,(\log t)^{2-2\sigma}+O\bigg(\dfrac{(\log t)^{2-2\sigma}}{(\sigma-\hh)(1-\sigma)^2}\bigg),
	\end{equation}
	and 
	\begin{equation}  
	\re\bigg(\dfrac{\zeta'}{\zeta}\bigg)'(\sigma+it)\geq -\bigg(\dfrac{-2\sigma^2+6\sigma-2}{\sigma(1-\sigma)}\bigg)\log\log t\,(\log t)^{2-2\sigma}+O\bigg(\dfrac{(\log t)^{2-2\sigma}}{(\sigma-\hh)(1-\sigma)^2}\bigg),
	\end{equation} 
	uniformly in the range
	\begin{align} \label{range2}
	\dfrac{1}{2}+\dfrac{\la_0}{\log\log t}\leq \sigma\leq   1 - \frac{c}{\sqrt{\log \log t}}  \quad \text{and} \quad t\geq 3,
	\end{align}
	 for any fixed $c>0$.
\end{theorem}

\smallskip

The main technique to prove these theorems revolves in bounding a certain sum over the ordinates of zeta-zeros
$$
\sum_{\ga} f(\ga-t),
$$
where $f$ is some explicit real function that varies according to the problem of study. The key idea is to replace $f$ by explicit bandlimited majorants and minorants that are in turn admissible for the Guinand-Weil explicit formula (Proposition \ref{GW}). From there estimating the sum is usually easier. This bandlimited approximation idea originates in the works of Beurling and Selberg (see \cite[Introduction]{V}), and was first employed in this form by Goldston and Gonek \cite{GG}, and Chandee and Soundararajan \cite{CS}, but many others after them (see \cite{CC, CCM, CChi, CChiM, Chi} to name a few). In our specific case, $f=f_a$ as in \eqref{fa}, which has \emph{zero mass} and therefore is not in the scope of the machinery developed in \cite{CLV}, nor its close relatives (the constructions in \cite{CLV} are regarded as the most general thus far and have been used widely).  Nevertheless, we are able to overcome this difficulty with a very simple optimal construction which, in the majorant case, requires some basic results in the theory of de Branges spaces.

We recall that, without assuming RH, explicit bounds for $\tfrac{\zeta'}{\zeta}(s)$ are given by Trudgian \cite{Trudgian} in a zero-free region for $\zeta(s)$.

\smallskip





\section{Lemmata}
For a given $a>0$ we let 
\begin{align}\label{fa}
f_{a}(x)=\dfrac{x^2-a^2}{\big(x^2+a^2\big)^2}.
\end{align}

\begin{lemma}[Representation lemma]\label{Rep_lem}
	Assume RH. We have
	\begin{equation*}
		\re\bigg(\dfrac{\zeta'}{\zeta}\bigg)'(\sigma+it)= \sum_\ga f_{\sigma-1/2}(\gamma-t) +  O\left(\dfrac{1}{t^2}\right),
		\end{equation*}
for $\hh < \sigma \leq 1$ and $t\geq 3$, where the above sum runs over the ordinates of the non-trivial zeros $\rho = \tfrac12 + i \gamma$ of $\zeta(s)$. 
\end{lemma}

\begin{proof}
Let $s=\sigma+it$ and $t \geq 3$. From the partial fraction decomposition for $\zeta'(s)/\zeta(s)$ (cf. \cite[Eq. 2.12.7]{Tit}), we have
	\begin{align}\label{partial_fraction_dec}
	\dfrac{\zeta'}{\zeta}(s) & = \displaystyle\sum_{\rho}\bigg(\dfrac{1}{s-\rho}+\frac{1}{\rho}\bigg)-\dfrac{1}{2}\dfrac{\Gamma'}{\Gamma}\bigg(\dfrac{s}{2}+1\bigg)+B+\frac{1}{2}\log\pi - \frac{1}{s-1},
	\end{align}
	with $B=-\sum_{\rho}\re(1/\rho)$. Differentiating and taking its real part we get
	$$
	\re\bigg(\dfrac{\zeta'}{\zeta}\bigg)'(\sigma+it)= \sum_{\gamma}f_{\sigma-1/2}(\gamma-t)-\dfrac{1}{4}\,\re\bigg(\dfrac{\Gamma'}{\Gamma}\bigg)'\Big(\dfrac{\sigma}{2}+1+\dfrac{it}{2}\Big)+O\left(\frac{1}{t^2}\right).
	$$
Using Stirling's formula, that guarantees the $\Gamma$ term is $O(1/t^2)$, we conclude.
\end{proof}

\smallskip

As always, the crucial tool to work with sums as in Lemma \ref{Rep_lem} is the Guinand-Weil explicit formula (see \cite[Lemma 8]{CChiM}), which for even functions reads as follows.

\begin{proposition}[Guinand-Weil explicit formula] \label{GW}
	Let $h(s)$ be analytic in the strip $|\im{s}|\leq \tfrac12+\varepsilon$, for some $\varepsilon>0$, such that $|h(s)|\ll(1+|s|)^{-(1+\delta)}$, for some $\delta>0$. Assume further that $h$ is even. Then
	\begin{align*}
	\displaystyle\sum_{\rho}h\left(\frac{\rho-\frac12}{i}\right) & = \dfrac{1}{2\pi}\int_{-\infty}^{\infty}h(u)\,\re{\dfrac{\Gamma'}{\Gamma}\left(\dfrac{1+2iu}{4}\right)}\,\du + 2\, h\left(\dfrac{i}{2}\right)-\dfrac{\log \pi}{2\pi}\widehat{h}(0) \\ &  \, \quad -\dfrac{1}{\pi}\displaystyle\sum_{n\geq2}\dfrac{\Lambda(n)}{\sqrt{n}}\widehat{h}\left(\dfrac{\log n}{2\pi}\right), 
	\end{align*}
	where $\rho = \beta + i \gamma$ are the non-trivial zeros of $\zeta(s)$ and
	$$
	\widehat{h}(y) = \int_{-\infty}^\infty h(x)e^{-2\pi i x y} \d x
		$$
		is the Fourier transform\footnote{We shall use this definition of the Fourier transform throughout the paper.} of $h$.
\end{proposition}

\subsection{Bandlimited approximations} 

\begin{lemma}[Minorant]\label{lemma:min} For $a,\Delta>0$ let
	\begin{equation}\label{Explicit_expression_min}
	L_{a,\Delta}(z) = \dfrac{z^2-a^2 - (Az^2+Ba^2)\sin^2(\pi\Delta z)}{(z^2+a^2)^2}
	\end{equation}
where
	\begin{align*}
	A = \dfrac{2\la \coth(\la)-1}{\sinh^2(\la)}, \quad B = \dfrac{2\la\coth(\la)+1}{\sinh^2(\la)},
	\end{align*}
and $\la=\pi a \Delta$. Then:
	\begin{enumerate}
	\item The inequality 
	$$L_{a,\Delta}(x) \leq f_{a}(x)$$
holds for all real $x$, $L_{a,\Delta}\in L^1(\R)$ and its Fourier transform is supported in $[-\Delta,\Delta]$ (i.e. $L_{a,\D}$ is of exponential type at most $2\pi \Delta$);
	
		\smallskip
		
		\item We have
		\begin{equation}\label{Poisson_intL}
	\ft  L_{a,\Delta}(0) = \frac{1}{\Delta}\sum_{n \in \Z}f_a(n/\Delta) = -\frac{\pi^2 \Delta}{\sinh^2(\pi a \Delta)},
	\end{equation}
	and any other function $F\neq L_{a,\Delta}$ having the same properties as $L_{a,\Delta}$ in item $(1)$ has integral strictly less than the integral of $L_{a,\Delta}$.
\end{enumerate}	
\end{lemma}

\begin{proof}
Note first that the constants $A,B$ were chosen so the numerator of $L_{a,\Delta}$ vanishes doubly at $z=\pm i a$. We then see that $L_{a,\Delta}$ is entire, of exponential type at most $2\pi \Delta$ and belongs to $L^1(\R)$. Therefore, the Paley-Wiener Theorem guarantees its Fourier transform is supported in $[-\D,\D]$. Since $B>A>0$ we have $L_{a,\Delta}(x)\leq f_a(x)$ for all real $x$. This proves item $(1)$. We now prove item $(2)$. Suppose $F$ is an $L^1(\R)$-function, $F(x)\leq f_a(x)$ for all real $x$ and $\ft F$ is supported in $[-\D,\D]$. Poisson summation implies
$$
\ft F(0) = \frac{1}{\D} \sum_{n\in\Z} F(n/\D) \leq  \frac{1}{\D} \sum_{n\in\Z} f_a(n/\D) = \frac{1}{\D} \sum_{n\in\Z} L_a(n/\D)=\ft L_{a,\D}(0),
$$
where the last identity is due to the fact that $L_{a,\Delta}$ interpolates (in second order) $f_a$ in $\tfrac{1}{\D}\Z$. Equality is attained if and only if $F(x)=L_{a,\D}(x)$ in second order for all $x\in \tfrac{1}{\D}\Z$. However, this completely characterizes $F=L_{a,\D}$ (see \cite[Theorem 9]{V}). Finally, using that $\ft f_a(y)=-2\pi^2|y|e^{-2\pi a |y|}$, identity \eqref{Poisson_intL} can easily be derived using Poisson summation over $\tfrac{1}{\D}\Z$.
\end{proof}

It turns out that because $f_a(x)$ has a local maximum at $x=\sqrt{3}\,a$, the bandlimited majorant of $f_a$ with minimal total mass will have to be adjusted when $\pi a \Delta$ is  small. This adjustment will require some de Branges spaces theory.

\begin{lemma}[Majorant]\label{lemma:maj}
For $a,\Delta>0$ let
\begin{equation}\label{Explicit_expression_maj}
U_{a,\Delta}(z)= \dfrac{z^2-a^2 + (Cz^2+Da^2)(\cos(\pi\Delta z)- E \pi \Delta z \sin(\pi \Delta z))^2}{(z^2+a^2)^2},
\end{equation}
where
\begin{align*}
(C,D,E) =\begin{cases} \quad \quad \ \ \bigg(\dfrac{2\la \tanh(\la)-1}{\cosh^2(\la)}, \dfrac{2\la\tanh(\la)+1}{\cosh^2(\la)}, 0 \bigg) & \  \text{if } \la \geq \la_0,  \vspace{3mm} \\ \bigg( 0,\dfrac{1}{2} \left(\dfrac{ 2\la +\tanh(\la)}{\sinh(\la)+\la \, {\rm sech}(\la)}\right)^2 , \dfrac{1-2\la \tanh(\la)}{2\la^2+ \la \tanh(\la)}\bigg) & \ \text{if } \la < \la_0, \end{cases}
\end{align*}
$\la=\pi a \Delta$ and $\la_0=0.771\ldots$ is such that $2\lambda_0 \tanh(\lambda_0)=1$. Then:
	\begin{enumerate}
	\item The inequality 
	$$f_{a}(x)\leq U_{a,\Delta}(x)$$
holds for all real $x$, $U_{a,\Delta}\in L^1(\R)$ and its Fourier transform is supported in $[-\Delta,\Delta]$ (i.e. $U_{a,\D}$ is of exponential type at most $2\pi \Delta$);
	
		\smallskip
		
		\item We have
		\begin{equation}\label{Poisson_int2case1}
	 \ft U_{a,\Delta}(0) = \begin{cases}\quad \quad \quad \quad \quad \quad \quad \quad \quad \quad \dfrac{\pi^2 \D}{\cosh^2(\la )} &  \text{if } \la \geq \la_0,  \vspace{2mm} \\  \dfrac{\pi^2 \D}{\sinh^2(\la)}\left( \dfrac{2\la+\sinh(2\la)}{8\la}\left(\dfrac{ 2\la +\tanh(\la)}{\sinh(\la)+\la \, {\rm sech}(\la)}\right)^2-1\right) &  \text{if } \la < \la_0. \end{cases}
	\end{equation}
Moreover, any other function $F\neq U_{a,\Delta}$ having the same properties as $U_{a,\Delta}$ in item $(1)$ has integral strictly greater than the integral of $U_{a,\Delta}$.
	\end{enumerate}
\end{lemma}

\begin{proof}
Note that the constants $(C,D,E)$ are chosen so that $U_{a,\Delta}$ is entire, that is, its numerator vanishes doubly at $z=\pm i a$. Since $U_{a,\Delta}$ is visibly of exponential type at most $2\pi \Delta$ and belongs to $L^1(\R)$, the Paley-Wiener Theorem guarantees its Fourier transform is supported in $[-\D,\D]$. Noting that $C,D \geq 0$ we have $f_a(x)\leq U_{a,\D}(x)$ for all real $x$, and this proves item $(1)$. We now show item $(2)$. Suppose $F$ is an $L^1(\R)$-function, $F(x)\geq f_a(x)$ for all real $x$ and $\ft F$ is supported in $[-\D,\D]$. We now apply the generalized Poisson summation formula of Littmann for bandlimited functions \cite[Theorem 2.1]{Litt}  for $\ga=(\pi E)^{-1}$ with $E>0$. It translates to
\begin{align*}
\ft F(0) & = \frac{1}{\Delta}\sum_{\B(t)=0} \biggl( 1-\frac{\pi E}{\pi(\pi^2 E^2t^2+1)+\pi E} \biggr)F(t/\Delta) \\ & \geq \frac{1}{\Delta}\sum_{\B(t)=0} \biggl( 1-\frac{\pi E}{\pi(\pi^2 E^2t^2+1)+\pi E} \biggr)f_a(t/\Delta) \\ & = \frac{1}{\Delta}\sum_{\B(t)=0} \biggl( 1-\frac{\pi E}{\pi(\pi^2 E^2t^2+1)+\pi E} \biggr)U_{a,\D}(t/\Delta) \\ & = \ft U_{a,\D}(0),
\end{align*}
where $\B(z)=\cos(\pi z)- E \pi  z \sin(\pi z)$. Note when $E=0$, that is, $\la\geq \la_0$, this is the classical Poisson summation over $\tfrac1{\D}(\tfrac12+\Z)$. Equality is attained if and only if $F(t/\D)=U_{a,\D}(t/\D)$ in second order for all real $t$ with $\B(t)=0$. We claim  this completely characterizes $F=U_{a,\D}$. The trick is to use the theory of de Branges spaces and the interpolation formula \cite[Theorem A]{GL} (the introduction of \cite{GL} gives a solid short background on the necessary de Branges spaces theory which we will use here without much explanation). First we note that the function $\E(z)=(i+\pi E z) e^{-\pi i z}$ is of Hermite-Biehler class (i.e. $|\E(\ov z)|<|\E(z)|$ for all $z$ with $\im z>0$) and therefore the de Branges space $\H(\E^2)$ exists, and it consists of all entire functions of exponential type at most $2\pi$ belonging to $L^2(\R,\dx/(1+E^2 \pi ^2 x^2))$. Note also that $\B(z)=i(\ov{\E(\ov z)}-\E(z))/2$. Moreover, it is not hard to show that all conditions of \cite[Theorem A]{GL} are satisfied by $\E(z)$, and thus we conclude that any function $G\in \H(\E^2)$ is completely characterized by its values $G(t)$ and $G'(t)$ for all real $t$ with $\B(t)=0$. Now it is simply a matter to note that $(i+\pi E z)^2F(z/\D)$ and $(i+\pi E z)^2 U_{a,\D}(z/\D)$ both belong to  $\H(\E^2)$, and so they must be equal\footnote{Note when $E=0$ this argument reduces to classical Paley-Wiener space theory and Poisson summation.}.

Finally, in the case $\la\geq \la_0$ one can use Poisson summation over  $\tfrac1{\D}(\tfrac12+\Z)$ to evaluate the integral of $U_{a,\D}$ and obtain
$$
\ft U_{a,\D}(0) = \frac{\pi^2 \Delta}{\cosh^2(\pi a \Delta)}.
$$
If $\la< \la_0$ then we can use Poisson summation over $\tfrac{1}{\D}\Z$ to obtain
\begin{align*}
\ft U_{a,\D}(0)  & = \frac{1}{\D}\sum_{n\in \Z} \left( f_a(n/\D) + \frac{Da^2}{(n^2/\D^2+a^2)^2}\right)  \\ 
& = -2\pi^2 \D \sum_{n\in \Z} |n|e^{-2\la |n|}  + D\pi^2\Delta \sum_{n\in\Z}  (|n|+\tfrac{1}{2\la})e^{-2\la |n|} \\ &  =  -\frac{\pi^2 \Delta}{\sinh^2(\la)} + D\pi^2\Delta  \frac{2\la+\sinh(2\la)}{4\la \sinh^2(\la)} \\
& = \dfrac{\pi^2 \D}{\sinh^2(\la)}\left( \frac{2\la+\sinh(2\la)}{8\la}\left(\dfrac{ 2\la +\tanh(\la)}{\sinh(\la)+\la \, {\rm sech}(\la)}\right)^2-1\right).
\end{align*}
Above we used that $\ft f_a(y)=-2\pi^2|y|e^{-2\pi a |y|}$ and the Fourier transform of $\tfrac{a^2}{(x^2+a^2)^2}$ is
${\pi^2}\left(|y|+\frac{1}{2\pi a}\right)e^{-2\pi a |y|}$.
\end{proof}

\begin{lemma}\label{lemma:majminbounds}
The functions defined in Lemmas \ref{lemma:min} and \ref{lemma:maj} satisfy the following inequalities  for $-\Delta < y < \Delta$:
$$
\ft L_{a,\Delta}(y)  < 0
$$
and, if $\pi a \D\geq \la_0$,
$$
\ft U_{a,\Delta}(y) > \ft f_a(y).
$$

\end{lemma}

\begin{proof}
First we deal with the minorant. Using that $\ft f_a(y)=-2\pi^2|y|e^{-2\pi a |y|}$ and the Fourier transforms of $\tfrac{1}{x^2+a^2}$ and $\tfrac{a^2}{(x^2+a^2)^2}$ are
$$
\frac{\pi}{a} e^{-2\pi a|y|} \quad \text{and} \quad {\pi^2}\biggl(|y|+\frac{1}{2\pi a}\biggr)e^{-2\pi a |y|},
$$
respectively, we obtain
\begin{align*}
 \ft L_{a,\Delta}(y)  =-2\pi^2|y|e^{-2\pi a |y|} - \frac{2\,{\rm Id}-T_\D - T_{-\D}}{4}\bigg[\pi^2\bigg(\frac{B+A}{2\pi a}  + (B-A)|y|\bigg)e^{-2\pi a |y|}\bigg],
 \end{align*}
where $T_{h}$ is the operator of translation by $h$ and ${\rm Id}$ is the identity operator. These operators come from the (distributional) Fourier transform of $\sin^2(\pi \D x)$. We claim that the function $e^{2\pi a y}\ft L_{a,\Delta}(y) $ is convex in the range $0<y<\Delta$, which would show that $\ft L_{a,\Delta}(y)$ is negative in the same range since it is negative at $y=0$ and vanishes at $y=\Delta$.  For $0<y<\D$ we have
\begin{align*}
 \frac{\d^2}{\d y^2} \left[e^{2\pi a y}\ft L_{a,\Delta}(y) \right] 
& =\frac{\d^2}{\d y^2} \left[\frac{\pi^2}4\left(\frac{B+A}{2\pi a} + { (B-A)(\D-y)}\right)e^{2\pi a (2y-\D)} + \text{linear}\right] \\
& =\left( A + { \pi a(B-A)(\D-y)}\right)4a\pi^3  e^{2\pi a (2y-\D)} \\ & > 0,
 \end{align*}
because $B>A>0$. The majorant case is simpler, since if  $\la=\pi a \D \geq  \la_0$ a similar computation leads to
\begin{align*}
\ft U_{a,\Delta}(y)  =-2\pi^2|y|e^{-2\pi a |y|} + \frac{2\,{\rm Id}+T_\D +T_{-\D}}{4}\bigg[\pi^2\bigg(\frac{C+D}{2\pi a}  + (D-C)|y|\bigg)e^{-2\pi a |y|}\bigg],
 \end{align*}
and so the desired inequality follows because $D>C\geq0$.
\end{proof}

\section{Proof of Theorem \ref{maintheorem}} Let $\hh<\sigma<1 $ and $\Delta>0$. Throughout the rest of the paper we set $a=\sigma-\hh$ and $\la=\pi a \D$. Using Lemma \ref{Rep_lem} and the evenness of the zeta-zeros we obtain
\begin{align*}
\re \bigg(\dfrac{\zeta'}{\zeta}\bigg)'(\sigma+it)  = \sum_{\ga} \left(\tfrac12 f_a(\ga-t) + \tfrac12 f_a(\ga+t) +f_a(\ga)\right) + O(1),
\end{align*}
as $t\to \infty$, where we have used that $f_a(x) = O(1/x^2)$ uniformly for $|x| \geq 1 $ and $0< a < 1/2$, hence $\sum_{\ga} f_a(\ga)=O(1)$. We then apply Lemmas \ref{lemma:min} and \ref{lemma:maj} to get
\begin{align}\label{Poisson_before_app_GW}	\displaystyle
& \sum_{\gamma}M_t L_{a,\Delta}(\gamma) +  O(1) \leq  \re\bigg(\dfrac{\zeta'}{\zeta}\bigg)'(\sigma+it)
\leq  \sum_{\gamma}M_t U_{a,\Delta}(\gamma) +  O(1),
\end{align}
where $M_t=\tfrac12 T_{t} +\tfrac12 T_{-t} + {\rm Id}$. Note that for each $t\geq 0$ the functions $M_t L_{a,\Delta}$ and $M_t U_{a,\Delta}$ are even and admissible for the Guinand-Weil explicit formula (Proposition \ref{GW}). We use the operator $M_t$ because its Fourier transform is the operator that multiplies by $2\cos^2(\pi t x)$, which is nonnegative. This will allow us to simply discard (or easily bound) the sum over primes in the explicit formula.

\subsection{Proof of the lower bound}
Applying Proposition \ref{GW} and  Lemmas \ref{lemma:min} and \ref{lemma:majminbounds} we obtain
\begin{align}\label{GW_applied_to_m}
\displaystyle\sum_{\gamma}M_t L_{a,\Delta}(\gamma)  &= \dfrac{1}{2\pi}\int_{-\infty}^{\infty}M_t L_{a,\Delta}(u)\,\re{\dfrac{\Gamma'}{\Gamma}\left(\dfrac{1+2iu}{4}\right)}\,\du + 2M_t L_{a,\Delta}\left(\dfrac{i}{2}\right)\\
& \,\,\quad
- \frac{2}{\pi}\sum_{n\geq 2}\dfrac{\Lambda(n)}{\sqrt{n}} \,\ft L_{a,\Delta}\left(\dfrac{\log n}{2\pi}\right) \cos^2(\tfrac12 t \log n)   -\dfrac{\log\pi }{\pi}\ft  L_{a,\Delta}(0)  \\ &
\geq  \dfrac{1}{2\pi}\int_{-\infty}^{\infty}M_t L_{a,\Delta}(u)\,\re{\dfrac{\Gamma'}{\Gamma}\left(\dfrac{1+2iu}{4}\right)}\,\du + 2M_t L_{a,\Delta}\left(\dfrac{i}{2}\right).
\end{align}
In this part we assume that $\la\geq c$ for some given fixed $c>0$. We now analyze the terms on the right-hand side above. The function $L_{a,\Delta}$ depends on the parameters $A$ and $B$, but both behave like (since $\la\geq c$)
$$
8\la e^{-2\la}+O(e^{-2\la}).
$$
Hence $|L_{a,\D}(x)| \leq K (x^2+a^2)^{-1}$ for some $K>0$. Since $(s^2+a^2)L_{a,\Delta}(s)$ has exponential type $2\pi \Delta$ and it is bounded on the real line, a routine application of the Phragm\'en–Lindel\"of principle implies that
\begin{align} \label{complexbound} 
|L_{a,\Delta}(s)|\leq K \dfrac{e^{2\pi \D |\im s|}}{|s^2+a^2|}, \quad s\in \C
\end{align}
(alternatively, one could derive such bound by direct computation). Using the bounds for $A$ and $B$ it follows that
\begin{align} 
2\, M_t L_{a,\Delta}\left(\dfrac{i}{2}\right)=\dfrac{4\pi a\Delta e^{(1-2a)\pi\Delta}}{a^2-\frac{1}{4}}+O\left(\dfrac{e^{(1-2a)\pi\Delta}}{(a^2-\frac{1}{4})^2}+\dfrac{e^{\pi\Delta}}{t^2}\right).
\end{align}Using that $M_t$ is self-adjoint and applying Stirling's approximation to obtain
$$
M_t\re\,\dfrac{\Gamma'}{\Gamma}\left(\dfrac{1+2iu}{4}\right)=\log t + O(\log(2+|u|)),
$$
we deduce that
\begin{align} 
\begin{split}
\dfrac{1}{2\pi}\int_{-\infty}^{\infty}M_t L_{a,\Delta}(u)\,\re\,\dfrac{\Gamma'}{\Gamma}\left(\dfrac{1+2iu}{4}\right) \du
&  = \dfrac{\pi\Delta\log t}{2\sinh^2(\pi a \Delta)} + O\left(\frac1a\right) \\
& = -2\pi \D e^{-2 \pi a \D}  \log t  + O\left(\tfrac1a+ \D  e^{-4\pi a \D}\log t\right) \\
& = -2\pi \D e^{-2 \pi a \D}  \log t  + O\left(\frac{1+  e^{-2\pi a \D}\log t}{a}\right).
\end{split}
\end{align}
Combining the above bounds we obtain
\begin{align} 
\begin{split}
\displaystyle\sum_{\gamma}M_tL_{a,\Delta}(\gamma-t) 
& \geq   -2\pi \D e^{-2 \pi a \D}  \log t  +\dfrac{4\pi a\Delta e^{(1-2a)\pi\Delta}}{a^2-\frac{1}{4}}  \\
& \,\,\,\,\, \,\,+ O\bigg(\dfrac{e^{(1-2a)\pi\Delta}}{(a^2-\frac{1}{4})^2}+\dfrac{e^{\pi\Delta}}{t^2}+ \frac1a\left(1+  e^{-2\pi a \D}\log t\right)\bigg).
\end{split}
\end{align}
Choosing $\pi\Delta=\log\log t$ (which is the optimal choice) and using \eqref{Poisson_before_app_GW} we obtain
	\begin{equation}  
\re\bigg(\dfrac{\zeta'}{\zeta}\bigg)'(\sigma+it)\geq -\bigg(\dfrac{-2\sigma^2+6\sigma-2}{\sigma(1-\sigma)}\bigg)\log\log t\,(\log t)^{2-2\sigma}+{O_{c}\left(\dfrac{(\log t)^{2-2\sigma}}{(\sigma-\hh)(1-\sigma)^2}\right)}
\end{equation} 
for $\pi(\sigma-1/2)\log\log t\geq c$. 
This proves {the desired result.} 
   \subsection{Proof of the upper bound} Using Proposition \ref{GW} and  Lemma \ref{lemma:maj} we obtain
\begin{align}
\displaystyle\sum_{\gamma}M_t U_{a,\Delta}(\gamma)  &\leq \dfrac{1}{2\pi}\int_{-\infty}^{\infty}M_t U_{a,\Delta}(u)\,\re{\dfrac{\Gamma'}{\Gamma}\left(\dfrac{1+2iu}{4}\right)}\,\du + 2M_t U_{a,\Delta}\left(\dfrac{i}{2}\right)\\
& \,\quad - \frac{2}{\pi}\sum_{n\geq 2}\dfrac{\Lambda(n)}{\sqrt{n}} \,\ft U_{a,\Delta}\left(\dfrac{\log n}{2\pi}\right) \cos^2 (\tfrac12 t \log n).
\end{align}
When $\la\geq \la_0$ the computations are very similar to the lower bound and we just indicate them here. We still have both $C$ and $D$ behaving like $8\la e^{-2\la}+O(e^{-2\la})$, and a bound similar to \eqref{complexbound} holds. Using Stirling's formula and Lemma \ref{lemma:maj} we get 
\begin{align} \label{15_352}
\begin{split}
\dfrac{1}{2\pi}\int_{-\infty}^{\infty}M_t U_{a,\Delta}(u)\,\re\,\dfrac{\Gamma'}{\Gamma}\left(\dfrac{1+2iu}{4}\right) \du & = \dfrac{\pi\Delta\log t}{2\cosh^2(\pi a \Delta)} + O\left(\frac1a\right) \\
& = 2\pi \D e^{-2 \pi a \D}  \log t  + O\left(\frac{1+  e^{-2\pi a \D}\log t}{a}\right).
\end{split}
\end{align}
Using the estimates for $C$ and $D$ it follows that
\begin{align} \label{0_162}
2M_t U_{a,\Delta}\left(\dfrac{i}{2}\right)=\dfrac{4\pi a\Delta e^{(1-2a)\pi\Delta}}{a^2-\frac{1}{4}}+O\bigg(\dfrac{e^{(1-2a)\pi\Delta}}{(a^2-\frac{1}{4})^2}\bigg) + O\left(\dfrac{e^{\pi\Delta}}{t^2}\right).
\end{align}
Since $\ft U_{a,\Delta}$ is supported in $[-\Delta, \Delta]$, we estimate the sum over primes (which we cannot discard as before) using Lemma \ref{lemma:majminbounds} and that $\ft f_a(y)=-2\pi^2 |y|e^{-2\pi a|y|}$ to get
\begin{align} 
-\frac{2}{\pi}\sum_{n\geq 2}\dfrac{\Lambda(n)}{\sqrt{n}} \,\ft U_{a,\Delta}\left(\dfrac{\log n}{2\pi}\right) \cos^2(\tfrac12 t \log n) & \leq  2 \sum_{2\leq n\leq e^{2\pi\Delta}}\dfrac{\Lambda(n)}{{n^{a+1/2}}} \log n \\
& = \dfrac{4\pi \Delta e^{(1-2a)\pi\Delta}}{\frac{1}{2}-a}+O\bigg(\dfrac{e^{(1-2a)\pi\Delta}}{(\frac{1}{2}-a)^2}+\dfrac{\Delta^3}{a}\bigg).
\end{align}
The above estimate follows from the prime number theorem (see \cite[Eq. (B.2)]{CChiM}). Choosing $\pi\Delta=\log\log t$ and using \eqref{Poisson_before_app_GW} we obtain
\begin{align}
\re\bigg(\dfrac{\zeta'}{\zeta}\bigg)'(\sigma+it) 
& \leq \bigg(\dfrac{-2\sigma^2+2\si+2}{\sigma(1-\sigma)}\bigg)\log\log t\,(\log t)^{2-2\sigma}+{O_{c}\bigg(\dfrac{(\log t)^{2-2\sigma}}{(\sigma-\hh)(1-\sigma)^2}\bigg)}
\end{align}
in the range $(\sigma-1/2)\log\log t\geq\lambda_0$ and $(1-\sigma)\sqrt{\log\log t} \geq c$ for some fixed $c>0$; note that $\Delta^3=O_c\left({(1/2-a)^{-2} }{(\log t)^{1-2a}}\right)$. This finishes the proof. \qed

\section{Proof of Theorem \ref{theorem1}} To obtain the bounds for the imaginary part of the log-derivative $\zeta(s)$ we will employ the interpolation technique of \cite[Section 6]{CChiM} for functions with slow growth, which we conveniently state in the form of a lemma.

\begin{lemma}[Interpolation]\label{lem:interp}
Let $\p: (t_0,\infty) \to \R$ be twice differentiable, $t_0>0$, and assume that 
$$
-\be_0(t) \leq \p(t) \leq \al_0(t) \quad \text{and} \quad -\be_2(t) \leq \p''(t) \leq \al_2(t),
$$ 
for some differentiable functions $\al_0,\be_0,\al_2,\be_2: (t_0,\infty) \to (0,\infty)$. Suppose the numbers
$$
L=\sup_{t>t_0} \, {\frac{2(\al_2(t)+\be_2(t))(\al_0(t)+\be_0(t))}{3\al_2(t)\be_2(t)}},
$$

$$M_i=\displaystyle\sup_{t>t_0}|\alpha'_i(t)| \ \ \ \ \text {and} \ \ \ \ \ N_i=\displaystyle\sup_{t>t_0}|\beta'_i(t)| \quad \quad (i=0,2)
$$
are finite. Then, for $t>t_0+\sqrt{3L}$ we have
$$
|\p'(t)| \leq \sqrt{\frac{2\al_2(t)\be_2(t)( \al_0(t)+\be_0(t))}{\al_2(t)+\be_2(t)}} + M_0+N_0+ (M_2+N_2)L.
$$
\end{lemma}

\begin{proof}
 Since the bound is symmetric when we interchange $\al_0,\be_0$ and $\al_2,\be_2$ (i.e. we change $\p$ by $-\p$), it is enough to prove that $\p'(t)$ is bounded above by the desired bound. An application of the  mean value theorem easily gives that\footnote{The notation $h_+$ means $\max(h,0)$.}
\begin{align}
\p'(t)- \p'(t-h) & = \p''(t^*)h \\ & \leq  h_+ \al_2(t^*) + (-h)_+ \be_2(t^*) \\ & \leq  h_+ \al_2(t) + (-h)_+ \be_2(t) + (M_2+N_2)|h|^2.`
\end{align}
Averaging in $h$ in the interval $[-\nu (1-A), \nu A]$, for some $\nu>0$ (with $t-\nu>t_0$) and $0<A<1$, we obtain
\begin{align*}
& \p'(t) \\  & \leq \tfrac1{\nu}(\p(t+(1-A)\nu) -\p(t-A\nu))  + \tfrac{\nu}{2}(A^2 \al_2(t) + (1-A)^2 \be_2(t)) + \tfrac{\nu^2}3(M_2+N_2) \\
& \leq  \tfrac1{\nu}(\al_0(t+(1-A)\nu) + \be_0(t-A\nu))  + \tfrac{\nu}{2}( A^2 \al_2(t) + (1-A)^2 \be_2(t)) +  \tfrac{\nu^2}3(M_2+N_2)\\
&\leq \tfrac1{\nu}(\al_0(t) + \be_0(t))  + \tfrac{\nu}{2}( A^2 \al_2(t) + (1-A)^2 \be_2(t)) + M_0+N_0+ \tfrac{\nu^2}3(M_2+N_2).
\end{align*}
Minimizing the main term above as a function of $\nu$ and $A$, we must set 
$$
\nu=\sqrt{\frac{2(\al_2(t)+\be_2(t))(\al_0(t)+\be_0(t))}{\al_2(t)\be_2(t)}} \quad \text{and} \quad A=\frac{\be_2(t)}{\al_2(t)+\be_2(t)},
$$
which gives, for $t>t_0+\sqrt{3L}$, that
$$
\p'(t) \leq \sqrt{\frac{2\al_2(t)\be_2(t)( \al_0(t)+\be_0(t))}{\al_2(t)+\be_2(t)}} +M_0+N_0+ (M_2+N_2)L.
$$
The lemma follows.
\end{proof}

We will apply this lemma for 
$$
\p(t)=-{\log |\zeta(\sigma+it)}|
$$
noting that
$$
\p'(t)=\im\dfrac{\zeta'}{\zeta}(\sigma+it) \quad \mbox{and} \quad \p''(t)=\re\bigg(\dfrac{\zeta'}{\zeta}\bigg)'(\sigma+it).
$$
Theorem \ref{maintheorem} and \cite[Theorems ~1 and ~2]{CC} establish respectively that
\begin{align} \label{secondderivative}
-\be_0(t)\leq  \p(t)\leq \al_0(t) \quad \text{and} \quad -\be_2(t) \leq \p''(t)\leq\al_2(t) ,
\end{align}
 in the range
\begin{align} \label{rangeinproof}
\dfrac{1}{2}+\dfrac{\la_0}{\log\log t} \leq  \sigma  \leq 1 - \frac{c\sqrt{\la_0/(\la_0+c)}}{\sqrt{\log \log t}}\quad \text{and} \quad t\geq 3,
\end{align}
where $c>0$, 
\begin{align*}
\al_2(t) & =\dfrac{-2\si^2+2\si+2}{\sigma(1-\sigma)}\,\ell_{-1,\sigma}(t)+O_c\left(\dfrac{\ell_{0,\sigma}(t)}{(\sigma-\hh)(1-\sigma)^2}\right), \\
\be_2(t) & = \dfrac{-2\si^2 + 6\si - 2}{\sigma(1-\sigma)}\,\ell_{-1,\sigma}(t)+O_c\left(\dfrac{\ell_{0,\sigma}(t)}{(\sigma-\hh)(1-\sigma)^2}\right), \\
\al_0(t) & = \be_0(t) = \dfrac{-\si^2 +5\si - 2}{2\sigma(1-\sigma)} \ell_{1,\sigma}(t) + O_c\left(\dfrac{\ell_{2,\sigma}(t)}{(1-\sigma)^2}\right)
\end{align*}
and $\ell_{n,\sigma}(t)=(\log t)^{2-2\si}(\log \log t)^{-n}$.
We can then apply Lemma \ref{lem:interp} with $t_0=t_0(\si,c)$ equals to the smallest $t$ such that \eqref{rangeinproof} is not vacuous. A routine computation shows that $\sqrt{3L}=O_c(1)$ and that $M_0, M_2, N_0,N_2$ are $O_c((\sigma-\hh)^{-1}(1-\sigma)^{-2})$. We obtain
\begin{align*}
\left|\im\dfrac{\zeta'}{\zeta}(\sigma+it)\right| & \leq  \sqrt{\frac{2\al_2(t)\be_2(t)( \al_0(t)+\be_0(t))}{\al_2(t)+\be_2(t)}} +O_c\left(\dfrac{1}{(\sigma-\hh)(1-\sigma)^2}\right) \\ 
&=  \sqrt{\frac{2(-\si^2 + 5\si - 2)(-\si^2 + 3\si - 1)(-\si^2 + \si + 1)}{\si^3(1-\si)^2(2-\si)}}\,\ell_{0,\si}(t) \\
& \quad \,+O_c\left(\dfrac{\ell_{1,\sigma}(t)}{(\sigma-\hh)(1-\sigma)^2}\right)
\end{align*}
if $t' = t-\sqrt{3L} \geq t_0$. Letting $t_1(c)$ be such that $\tfrac{\log \log t'}{\log \log t}\geq \tfrac{\la_0}{\la_0+c}$ if $t\geq t_1(c)$, we conclude that the above estimate holds in the range
\begin{align} \label{rangeinproof2}
\dfrac{1}{2}+\dfrac{\la_0+c}{\log\log t} \leq  \sigma  \leq 1 - \frac{c}{\sqrt{\log \log t}}\quad \text{and} \quad t\geq t_2(c),
\end{align}
where $t_2(c)=\max(t_1(c),3+\sqrt{3L})$. To finish the proof we note that if $3\leq t \leq t_2(c)$ then a simple compactness argument gives the full desired range. \qed


\section*{Acknowledgements}
AC was supported by Grant 275113 of the Research Council of Norway. The authors are grateful to Kristian Seip for his helpful remarks and suggestions.


\begin{thebibliography}{9999}


\bibitem{CC} 
E. Carneiro and V. Chandee, 
\newblock Bounding $\zeta(s)$ in the critical strip, 
\newblock J. Number Theory 131 (2011), 363--384.


\bibitem{CCM}
E. Carneiro, V. Chandee and M. B. Milinovich, 
\newblock Bounding $S(t)$ and $S_1(t)$ on the Riemann hypothesis,
\newblock Math. Ann. 356 (2013), no. 3, 939--968.




\bibitem{CChi} 
E. Carneiro and A. Chirre,
\newblock Bounding $S_n(t)$ on the Riemann hypothesis,
\newblock Math. Proc. Cambridge Philos. Soc. 164 (2018), no. 2, 259--283.


\bibitem{CChiM}  E. Carneiro, A. Chirre and M. B. Milinovich,
\newblock Bandlimited approximations and estimates for the Riemann zeta-function,
\newblock Publ.  Mat.  63 (2019), no.  2, 601--661.









\bibitem{CLV}
E. Carneiro, F. Littmann and J. D. Vaaler, 
\newblock Gaussian subordination for the Beurling-Selberg extremal problem, 
\newblock Trans. Amer. Math. Soc. 365, no. 7 (2013), 3493--3534.


\bibitem{CS} 
V. Chandee and K. Soundararajan,
\newblock Bounding $|\zeta(\frac{1}{2}+it)|$ on the Riemann hypothesis, 
\newblock Bull. London Math. Soc. 43 (2011), no. 2, 243--250.

\bibitem{Chi}  A. Chirre,
\newblock A note on Entire L-functions,
\newblock Bull. Braz. Math. Soc. (N.S.) 50 (2019), no.1, 67--93.






\bibitem{GG} 
D. A. Goldston and S. M. Gonek,
\newblock A note on $S(t)$ and the zeros of the Riemann zeta-function, 
\newblock Bull. London Math. Soc. 39 (2007), 482--486.


\bibitem{GL}
F. Gon\c{c}alves and L. Littmann,
\newblock Interpolation formulas with derivatives in de Branges spaces II,
\newblock J. Math. Anal. Appl. 458 (2018), no. 2, 1091--1114. 










\bibitem{Ki}
H. Ki,
\newblock The zeros of the derivative of the Riemann zeta function near the critical line,
\newblock Int. Math. Res. Not. IMRN (2008), no. 16, Art. ID rnn064, 23 pp. 


\bibitem{Litt}
F. Littmann,
\newblock Quadrature and extremal bandlimited functions, 
\newblock SIAM J. Math. Anal. 45 (2013), no. 2, 732--747.










\bibitem{Tit}
E. C. Titchmarsh, 
\newblock The theory of the Riemann zeta-function, Second edition, \newblock The Clarendon Press, Oxford University Press, New York, 1986.

\bibitem{Trudgian}
T. Trudgian,
\newblock Explicit bounds on the logarithmic derivative and the reciprocal of the Riemann zeta-function,
\newblock Funct. Approx. Comment. Math. 52 (2015), 253–261.


\bibitem{V}
J. D. Vaaler,
\newblock Some extremal functions in Fourier analysis,
\newblock Bull. Amer. Math. Soc. 12 (1985), 183--216.















\end{thebibliography}
\end{document}